\documentclass[a4paper,12pt,final]{amsart}
\usepackage{times,a4wide,mathrsfs,amssymb}

\newcommand{\C}{\mathbb{C}}
\newcommand{\ZZ}{\mathbb{Z}}

\newcommand{\QQ}{\mathbb{Q}}
\newcommand{\NN}{\mathbb{N}}
\newcommand{\PP}{\mathbb{P}}
\newcommand{\A}{\mathbb{A}}
\newcommand{\OO}{\mathcal O}
\newcommand{\Ss}{\mathcal S}
\newcommand{\DD}{\mathcal D}

\newcommand{\YY}{\mathcal Y}

\newcommand{\VV}{\mathcal V}

\newcommand{\Z}{\mathcal Z}
\newcommand{\EE}{\mathcal E}
\newcommand{\MM}{\mathcal M}

\newcommand{\gr}{\hbox{Gr}}
\newcommand{\wt}{\widetilde}
\newcommand{\ima}{\hbox{Im}}
\newcommand{\rom}{\romannumeral}

\newtheorem{theorem}{Theorem}[section]

\newtheorem{lemma}[theorem]{Lemma}
\newtheorem{corollary}[theorem]{Corollary}
\newtheorem{proposition}[theorem]{Proposition}
\newtheorem{conjecture}[theorem]{Conjecture}
\newtheorem{remark}[theorem]{Remark}
\newtheorem{definition}[theorem]{Definition}
\newtheorem{convention}{Conventions}

\newtheorem{nonumbering}{Theorem}

\newtheorem{nonumberingp}{Proposition}

\newtheorem{nonumberingt}{Acknowledgements}

\begin{document}
\author[Robert Laterveer]
{Robert Laterveer}

\address{Institut de Recherche Math\'ematique Avanc\'ee, Universit\'e 
de Strasbourg, 7 Rue Ren\'e Des\-car\-tes, 67084 Strasbourg CEDEX, France.}
\email{robert.laterveer@math.unistra.fr}

\title{Algebraic cycles and Todorov surfaces}

\begin{abstract} Motivated by the Bloch--Beilinson conjectures, Voisin has formulated a conjecture about $0$--cycles on self--products of surfaces of geometric genus one. We verify Voisin's conjecture for the family of Todorov surfaces with $K^2=2$ and fundamental group $\ZZ/2\ZZ$. As a by--product, we prove that certain Todorov surfaces have finite--dimensional motive.
\end{abstract}



\maketitle

\section{Intro}

The Bloch--Beilinson conjectures have been hugely influential in making concrete predictions concerning the behaviour of Chow groups with $\QQ$--coefficients $A^\ast()_{\QQ}$ of smooth projective varieties over $\C$ (this is explained, for example, in \cite{Vo}, \cite{MNP}, \cite{J2}).
One of these concrete predictions is the following intriguing conjecture about $0$--cycles on self--products of surfaces with geometric genus one:

\begin{conjecture}[Voisin \cite{V9}]\label{conj} Let $S$ be a smooth complex projective surface with $h^{0,2}(S)=1$ and $q(S)=0$. Let $a,a^\prime\in A^2_{hom}(S)$ be two $0$--cycles of degree $0$. Then
  \[ a\times a^\prime= a^\prime\times a\ \ \ \hbox{in}\ A^4(S\times S)\ .\]
   (The notation $a\times a^\prime$ is a short--hand for the cycle class $(p_1)^\ast (a)\cdot (p_2)^\ast(a^\prime)\in A^{4}(S\times S)$, where $p_1, p_2$ denote projection on the first, resp. second factor.)  
 \end{conjecture}
  
Conjecture \ref{conj} has been verified in certain cases \cite{V9}, \cite{thoughts}, but is still wide open for a general $K3$ surface.\footnote{More precisely: I am not aware of a single $K3$ surface with Picard number $<9$ for which Conjecture \ref{conj} is known.}  

The principal aim of this note is to add some new items to the list of examples of surfaces for which Conjecture \ref{conj} is verified. The main result is as follows:

\begin{nonumbering}[=Corollary \ref{true}] Let $S$ be a Todorov surface with $K^2_S=2$ and $\pi_1(S)=\ZZ/ 2\ZZ$. Then Conjecture \ref{conj} is true for $S$.
\end{nonumbering}

A Todorov surface (cf. Definition \ref{tod} below for a precise definition) is a certain surface of general type, for which the bicanonical map factors over a $K3$ surface; these surfaces have been intensively studied with the aim of providing counterexamples to local and global Torelli \cite{Kunev}, \cite{Tod2}, \cite{Mor3}, \cite{U}, \cite{U2}. There exist $11$ irreducible families of Todorov surfaces \cite{Mor3}. Todorov surfaces with invariants $K^2_S=2$ and $\pi_1(S)=\ZZ/2\ZZ$ form one of these irreducible families, which is of dimension $12$. In \cite{thoughts}, I established the truth of Conjecture \ref{conj} for another irreducible family of Todorov surfaces (those with $K^2_S=1$, which are sometimes called ``Kunev surfaces''); so now there remain $9$ more families to investigate.

Along the way, we obtain some other results that may be of independent interest. For example, the above result is obtained by first showing the following:

\begin{nonumbering}[=Theorem \ref{main2}] Let $S$ be a Todorov surface with $K^2_S=2$ and $\pi_1(S)=\ZZ/ 2\ZZ$, and let $P$ be the $K3$ surface associated to $S$.
There is an isomorphism of Chow motives
  \[  t_2(S)\ \cong\ t_2(P)\ \ \ \hbox{in}\ \MM_{\rm rat}\ \]
  (here $t_2$ denotes the transcendental part of the motive \cite{KMP}).
\end{nonumbering}

This has consequences for the intersection product on $S$ (Corollary \ref{intersection}).
As another consequence of Theorem \ref{main2}, we are able to show (Corollary \ref{finite}) that certain Todorov surfaces have finite--dimensional motive in the sense of Kimura and O'Sullivan \cite{K}, \cite{An}. This provides some new examples of surfaces of general type with finite--dimensional motive. The proof of Theorem \ref{main2} is directly inspired by Voisin's work on the Bloch/Hodge equivalence for complete intersections \cite{V0}, \cite{V1}, reasoning family--wise and using the technique of ``spread'' of algebraic cycles.

\vskip0.6cm

\begin{convention} In this note, the word {\sl variety\/} will refer to a quasi--projective separated scheme of finite type over $\C$, endowed with the Zariski topology. A {\sl subvariety\/} is a (possibly reducible) reduced subscheme which is equidimensional. 

We will denote by $A_j(X)$ the Chow group of $j$--dimensional cycles on $X$;
for $X$ smooth of dimension $n$ the notations $A_j(X)$ and $A^{n-j}(X)$ will be used interchangeably. 
Chow groups with rational coefficients will be denoted
  \[ A_j(X)_{\QQ}:=A_j(X)\otimes_{\ZZ} \QQ\ .\]
The notation $A^j_{hom}(X)$, resp. $A^j_{AJ}(X)$ will be used to indicate the subgroups of homologically trivial, resp. Abel--Jacobi trivial cycles.
For a morphism $f\colon X\to Y$, we will write $\Gamma_f\in A_\ast(X\times Y)$ for the graph of $f$.



In an effort to lighten notation, we will write $H^j(X)$ (or $H_jX$) to indicate singular cohomology $H^j(X,\QQ)$ (resp. Borel--Moore homology $H_j(X,\QQ)$).

\end{convention}

\section{Todorov surfaces}

This preparatory section contains the definition and basic properties of Todorov surfaces. A first result that will be crucial to us is that any Todorov surface has an associated $K3$ surface for which Voisin's conjecture is known to hold (Theorem \ref{rito}; this is work of Rito). A second crucial result is that Todorov surfaces with $K_S^2=2$ and $\pi_1(S)=\ZZ/2\ZZ$ can be described as quotients of certain complete intersections in a weighted projective space (Theorem \ref{cd}; this is work of Catanese--Debarre).

\begin{definition}[\cite{Mor3}, \cite{Tod2}]\label{tod} A {\em Todorov surface\/} is a smooth projective surface $S$ of general type with $p_g(S)=1$, $q=0$, and such that the bicanonical map $\phi_{2K_S}$ factors as
  \[ \phi_{2K_S}\colon\ \  S  \ \xrightarrow{\iota}\ S\ \dashrightarrow\ \PP^r\ ,\]
  where $\iota\colon S\to S$ is an involution for which $S/\iota$ is birational to a $K3$ surface (i.e., there is equality $\phi_{2K_S} \circ \iota =  \phi_{2K_S}$).
  
  The $K3$ surface obtained by resolving the singularities of $S/\iota$ will be called the $K3$ surface {\em associated to $S$\/}.
\end{definition}

\begin{definition}[\cite{Mor3}]\label{fi} The {\em fundamental invariants\/} of a Todorov surface $S$ are $(\alpha,k)$, where $\alpha$ is such that the $2$--torsion subgroup of $\hbox{Pic}(S)$ has order $2^\alpha$, and $k=K_S^2+8$.
\end{definition}

\begin{remark}\label{remtod} Morrison proves \cite[p. 335]{Mor3} there are exactly $11$ non--empty irreducible families of Todorov surfaces, corresponding to the $11$ possible values of the fundamental invariants:
  \[ \begin{split}(\alpha,k)\in\ \ \Bigl\{ (0,9),(0,10),&(0,11),(1,10),(1,11),\\
      &(1,12),(2,12),(2,13),(3,14),(4,15),(5,16)\Bigr\}\ .\\
      \end{split}\]
Examples of surfaces belonging to each of the $11$ families are given in \cite{Tod2}; moreover, it is shown in loc. cit. that these surfaces provide counterexamples to local and global Torelli (cf. \cite{U2} for an overview on Torelli problems, and \cite{U} where a mixed version of Torelli is proposed to remedy this failure).
The family with fundamental invariants $(0,9)$ was first described by Kunev \cite{Kunev}; these surfaces are sometimes called {\em Kunev surfaces\/}.

In \cite{Mor3}, an explicit description is given of the coarse moduli space for each of the $11$ families of Todorov surfaces.

Lee and Polizzi have given an alternative construction of Todorov surfaces, as deformations of product--quotient surfaces \cite[Theorem 4.6 and Remark 4.7]{LP}.
\end{remark}

\begin{remark} The convention $k=K_S^2+8$ in Definition \ref{fi}, which may appear strange at first sight, is explained as follows: the number $k$ happens to be the number of rational double points on a so--called ``distinguished partial desingularization'' of $S/\iota$ (this follows from \cite[Theorem 5.2 (\rom2)]{Mor3}).
\end{remark}

We will make use of the following result:

\begin{theorem}[Rito \cite{Rito}]\label{rito} Let $S$ be a Todorov surface, and let $P$ be the smooth minimal model of $S/\iota$. Then there exists a generically finite degree $2$ cover
  \[ P\ \to\ \PP^2\ ,\]
  ramified along the union of two smooth cubics.
\end{theorem}

\begin{remark} For the Todorov surface with fundamental invariants $(0,9)$ (aka a Kunev surface), Theorem \ref{rito} was already proven by Kunev and Todorov \cite{Tod}.
\end{remark}

We now restrict attention to Todorov surfaces $S$ with fundamental invariants $(1,10)$. This means that $K_S^2=2$ and (according to \cite[Theorem 2.11]{CD}) the fundamental group of $S$ is $\ZZ/2\ZZ$. In this case, there happens to be a nice explicit description of $S$ in terms of weighted complete intersections:

\begin{theorem}[Catanese--Debarre \cite{CD}]\label{cd} Let $S$ be a Todorov surface with fundamental invariants $(1,10)$. Then the canonical model of $S$ is the quotient $V/\tau^\prime$, where $V\subset\PP(1^3,2^2)$ is a weighted complete intersection having only rational double points as singularities, given by the equations
  \[ \begin{cases} F=z_3^2+c w^4+w^2 q(x_1,x_2)+Q(x_1,x_2)=0\ ,&\\
                         G=z_4^2+c^\prime w^4+w^2 q^\prime(x_1,x_2)+ Q^\prime(x_1,x_2)=0\ .&\\
                       \end{cases}  \]
    Here $[w:x_1:x_2:z_3:z_4]$ are coordinates for $\PP:=\PP(1^3,2^2)$, and $q,q^\prime$ are quadratic forms, $Q,Q^\prime$ are quartic forms without common factor, and $c,c^\prime$ are constants not both $0$. The involution $\tau^\prime\colon \PP\to \PP$ is defined as 
      \[ [w:x_1:x_2:z_3:z_4]\ \mapsto\ [-w:x_1:x_2:z_3:z_4]\ .\]
  Conversely, given a weighted complete intersection $V\subset\PP$ as above, the quotient $V/\tau^\prime$ is the canonical model of a Todorov surface with fundamental invariants $(1,10)$.    
 \end{theorem}

\begin{proof} This is a combination of \cite[Theorem 2.8]{CD} and \cite[Theorem 2.9]{CD}.
\end{proof}

\begin{remark} The focus in the paper \cite{CD} is not on Todorov surfaces as such, but rather (as the title indicates) on {\em all\/} surfaces of general type with $p_g=1$, $q=0$ and $K^2=2$. Theorem \ref{cd} is actually a special case of the more general \cite[Theorem 2.9]{CD}, which describes the canonical model of {\em all\/} surfaces with $p_g=1$, $q=0$, $K^2=2$ and $\pi_1=\ZZ/2\ZZ$ as quotients of weighted complete intersections. As shown in loc. cit., such surfaces form a $16$--dimensional irreducible family. The Todorov surfaces with fundamental invariants $(1,10)$ correspond to  surfaces with these invariants and for which the bicanonical map is a Galois covering; they
 form a $12$--dimensional subfamily inside this $16$--dimensional family.
 
 The same remark can be made about Todorov surfaces with fundamental invariants $(0,9)$ (aka ``Kunev surfaces''): these form a $12$--dimensional subfamily inside the (18--dimensional) family of all surfaces of general type with $p_g=1$, $q=0$ and $K^2=1$; this family (and the $12$--dimensional subfamily of Kunev surfaces, corresponding to the bicanonical map being Galois) can also be explicitly described in terms of weighted complete intersections \cite{Cat}, \cite{Tod}.
 \end{remark}
 
 \begin{remark} Todorov surfaces appear as so--called ``non--standard cases'' in the classification of surfaces of general type whose bicanonical map fails to be birational \cite[Chapter 2]{BCP}. The Todorov surfaces with fundamental invariants $(1,10)$ appear as item (\rom4) of \cite[Theorem 8]{BCP}, the Kunev surfaces show up as item (\rom3) and the other Todorov surfaces are covered by item (\rom5) of \cite[Theorem 8]{BCP}.
 \end{remark}

It will be convenient to rephrase Theorem \ref{cd} as follows:

\begin{corollary}\label{family} Let $\PP$ be the weighted projective space $\PP:=\PP(1,1,1,2,2)$. Let 
  \[ B\subset \bigl(\PP H^0(\PP, {\mathcal O}_{\PP}(4))\bigr)^{\times 2}\]
  denote the subspace parametrizing pairs of weighted homogeneous equations of type
     \[ \begin{cases} F_b=z_3^2+c w^4+w^2 q(x_1,x_2)+Q(x_1,x_2)=0\ ,&\\
                         G_b=z_4^2+c^\prime w^4+w^2 q^\prime(x_1,x_2)+ Q^\prime(x_1,x_2)=0\ ,&\\
                       \end{cases}  \]
                   where $(F_b,G_b)$ is as in Theorem \ref{cd}, i.e. the variety
              \[ V_b:= \bigl\{  x\in \PP\ \vert\ F_b(x)=G_b(x)=0 \bigr\}    \]
           has only rational double points as singularities. (Thus, $B$ is a Zariski open in a product of projective spaces $\bar{B}=   \PP^{r}\times \PP^{r}$, parametrizing all equations of type $(F_b,G_b)$, without conditions on the singularities.)    
           
        Let
        \[ \VV\ \to\ B\]
        denote the total space of the family (i.e., the fibre over $b\in B$ is the variety $V_b\subset\PP$), and let
        \[ \Ss:=\VV/\tau \ \to\ B\]
        denote the family obtained by applying the (fixed point--free) involution $\tau:=\tau^\prime\times\hbox{id}_B$ to $\Ss\subset \PP\times B$ . Then $\Ss\to B$ is the family of all canonical models of Todorov surfaces with fundamental invariants $(1,10)$.
        
           \end{corollary}
           

\begin{proposition}\label{smooth} The quasi--projective varieties $\VV$ and $\Ss$ defined in Corollary \ref{family} are smooth.
\end{proposition}

\begin{proof} 

We first establish a preparatory lemma:

 \begin{lemma}\label{one} For each point
       \[x\in\PP\setminus  [0:0:0:1:0]\]
       there exists a polynomial $G_b$ as in corollary \ref{family} such that
       \[ x\not\in \ \ (G_b=0)\ .\]
       
      For each point
        \[x\in \PP \setminus [0:0:0:0:1])\ ,\] 
       there exists a polynomial $F_b$ as in corollary \ref{family} such that
       \[ x\not\in \ \ (F_b=0)\ .\]
        \end{lemma}    
     
     \begin{proof} If $x\in\PP$ is different from $[0:0:0:1:0]$, consider the image of $x$ under the projection
     \[  \PP\setminus [0:0:0:1:0]\ \to\ \PP(1,1,1,2)\ ,\]
     given by forgetting the $z_3$ coordinate. It is easily seen that the linear system defined by the $G_b$ on $\PP(1,1,1,2)$ is base--point--free.
     
     Likewise, for $x\in\PP$ different from $[0:0:0:0:1]$, consider the projection
      \[  \PP\setminus [0:0:0:0:1]\ \to\ \PP(1,1,1,2)\ ,\]
     given by forgetting the $z_4$ coordinate.      
       \end{proof}
     
 Consider now $\bar{B}=\PP^r\times\PP^r$ the projective closure of $B$, parametrizing complete intersections that may be badly singular. Let 
   \[  \bar{\VV} \ \subset \bar{B}\times\PP \]
   denote the incidence variety containing $\VV$ as an open subset, and let $\pi\colon\bar{\VV}\to\PP$ denote the morphism induced by projection. Lemma \ref{one} says that for any point 
     \[ p\in\PP\setminus ([0:0:0:1:0]\cup    [0:0:0:0:1]  \ ,\]
     the fibre over $p$ is
     \[ \pi^{-1}(p)\cong \PP^{r-1}\times\PP^{r-1}\ .\]
     It follows that the quasi--projective variety
     \[  \bar{\VV}_{reg}:= \pi^{-1}(\PP_{reg})\ ,\]
     being a projective bundle over a projective bundle over the smooth variety $\PP_{reg}$, is smooth.
     
     But the singular locus of $\PP$ is exactly the line $w=x_1=x_2=0$, and a direct verification shows that $V_b$ as in Corollary \ref{family} does not meet this singular line, i.e. $V_b\subset\PP_{reg}$ for each $b\in B$ and hence
     \[ \VV\ \subset\ \bar{\VV}_{reg}\ .\]
   This proves smoothness of $\VV$. The smoothness of $\Ss$ now follows since $\Ss$ is the quotient of $\VV$ under a fixed point--free involution.  
      \end{proof}    
      
    \begin{corollary}\label{genfibre} The general $V_b$ and the general $S_b$ are smooth.
    \end{corollary} 
    
    \begin{remark} Corollary \ref{genfibre} is also established (by a different argument) in \cite[Remark 2.10]{CD}.
    \end{remark}

\section{Main result}

In this section, the main result as announced in the introduction (Theorem \ref{main}) is reduced to a statement concerning the Chow group of codimension $2$ cycles on the relative self--product of a family (Proposition \ref{key}). 
This reduction step is done by reasoning family--wise, using the method of ``spread'' of algebraic cycles developed by Voisin in her work on the Bloch/Hodge equivalence \cite{V0}, \cite{V1}, \cite{Vo}.
The proof of Proposition \ref{key} is postponed to section \ref{seckey}.

\begin{theorem}\label{main} Let $S$ be a Todorov surface with $K^2_S=2$ and $\pi_1(S)=\ZZ/ 2\ZZ$, and let $P$ be the $K3$ surface obtained as a resolution of singularities of $S/\iota$.
The natural correspondence from $S$ to $P$ induces an isomorphism 
  \[  A^2_{hom}(S)_{\QQ}\ \cong\ A^2_{hom}(P)_{\QQ}\ .\]
\end{theorem}

Theorem \ref{main} implies the truth of Voisin's conjecture (Conjecture \ref{conj}) for $S$:

\begin{corollary}\label{true} Let $S$ be a Todorov surface with $K^2_S=2$ and $\pi_1(S)=\ZZ/2\ZZ$. Let $a,a^\prime\in A^2_{hom}(S)$ be two $0$--cycles of degree $0$. Then
  \[ a\times a^\prime= a^\prime\times a\ \ \ \hbox{in}\ A^4(S\times S)\ .\]
\end{corollary}

\begin{proof} Since (by Rojtman's theorem \cite{R}) there is no torsion in $A^4_{hom}(S\times S)$, it suffices to prove the statement with rational coefficients. Let $P$ be the $K3$ surface obtained by resolving the singularities of $S/\iota$. There is a commutative diagram
  \[ \begin{array}[c]{ccc}
         A^2_{hom}(S)_{\QQ}\otimes A^2_{hom}(S)_{\QQ} &\to & A^4(S\times S)_{\QQ}\\
         \uparrow&&\uparrow\\
         A^2_{hom}(P)_{\QQ}\otimes A^2_{hom}(P)_{\QQ} &\to & A^4(P\times P)_{\QQ}\\ 
         \end{array}\]
         
   Here the left vertical arrow is an isomorphism (Theorem \ref{main}). The $K3$ surface $P$ admits a description as a blow--up of a double cover of $\PP^2$ branched along $2$ cubics (Theorem \ref{rito}). It follows that Voisin's conjecture is true for $P$, i.e. any $b,b^\prime  \in A^2_{hom}P$ satisfy
   \[  b\times b^\prime= b^\prime\times b\ \ \hbox{in}\ A^4(P\times P)\ ;\]    
   this is proven by Voisin \cite[Theorem 3.4]{V9}. This implies Voisin's conjecture is true for $S$.
   \end{proof}         

We now proceed to prove Theorem \ref{main}:

\begin{proof}(of Theorem \ref{main}) (This proof is directly inspired by Voisin's work on the Bloch/Hodge equivalence \cite{V0}, \cite{V1}, \cite{Vo}.)

The work of Catanese--Debarre (\cite{CD}, theorem \ref{cd}) implies that canonical models of Todorov surfaces with fundamental invariants $(1,10)$ form a family
  \[ \Ss\ \to\ B\]
  as in Corollary \ref{family}. Moreover, there exist morphisms of families over $B$
  \[  \VV\ \to\ \Ss\ \xrightarrow{f}\ \MM \ \to\ \EE\ \to\ B\ ,\]
  where $\EE$ is the family of quadric cones in $\PP^3$ (the quadric cone $\EE_b$ is the image of $\Ss_b$ under the bicanonical map \cite{CD}), and $\MM$ is the family of
  $K3$ surfaces with rational double points. Recall from corollary \ref{family} that $\Ss=\VV/\tau$ where $\tau$ is an involution, and $\MM=\Ss/\iota$ where $\iota$ is an involution. As explained in \cite[Remark 2.10]{CD}, the family $\MM$ can be obtained from the family $\EE$ by taking a double cover with prescribed ramification, and the family $\Ss$ is obtained from $\MM$ by taking a double cover, and the same for $\VV$ over $\Ss$.
  
  To be on the safe side, we prefer to resolve singularities and work with smooth varieties.
  That is, we construct a commutative diagram of families over $B$
    \[ \begin{array}[c]{ccc}  \wt{\VV}& \to& \VV \\
                              \downarrow&& \downarrow\\
                                   \wt{\Ss}& \to & \Ss\\
                                  \ \ \ \  \downarrow{\wt{f}}&& \ \ \downarrow{f}\\
                                  \wt{\MM}&\to& \MM\\
                                  \downarrow&&\ \ \downarrow{g}\\
                                  \wt{\EE}&\to& \EE\\
                                  &\searrow\ &\downarrow\\
                                  &&B\\
                                  \end{array}\]
                    where varieties in the left column are smooth.
                    This is not harmful to the argument, thanks to the following lemma:
                    
                 \begin{lemma}\label{desing}
                 
                 For any $b\in B$, the induced morphisms
                 \[ \wt{\VV}_b\to \VV_b\ , \ \ \wt{S}_b\to S_b\ ,\ \ \wt{\MM}_b\to \MM_b\  ,\ \  
                  \wt{\EE}_b\to \EE_b
                  \]
                 are birational.
                  \end{lemma}        
                  
              \begin{proof} As noted above (Proposition \ref{smooth}), $\Ss$ and $\VV$ are smooth.
              It follows that the singular locus of $\MM$ consists of the image of the fixed locus of the involution associated to $f$.
              Likewise, the singular locus of $\EE$ consists of the image of the singular locus of $\MM$, plus the image of the fixed locus of the involution associated to $g$.
              Since the involutions associated to $f$ and $g$ restrict to an involution on each fibre, we have
                \[ \dim\Bigl(\hbox{Sing}(\EE)\cap \EE_b\Bigr)\le 1\ \ \hbox{for\ all\ }b\in B\ .\]
                This implies the induced morphism 
                \[   \wt{\EE}_b\ \to\ \EE_b \]
                is birational for all $b\in B$.
                
                The variety $\wt{\MM}$ is obtained by resolving the singularities of the fibre product $\wt{\EE}\times_{\EE} \MM$. Since the open subset $\EE_{\rm reg}$ meets every fibre $\EE_b$,
                and $g$ restricts to a smooth morphism over $\EE_{\rm reg}$, the morphism $\wt{\MM}\to \MM$ is an isomorphism over the open $g^{-1}(\EE_{\rm reg})$. This open subset meets all the fibres $\MM_b$, and so
                \[ \wt{\MM}_b\ \to\ \MM_b\]
                is birational for all $b\in B$.
                
            The argument for $\Ss$ and $\VV$ is the same.

              \end{proof}

  We will be interested in the family
  \[   \wt{\Ss}\times_B \wt{\Ss}\ \to\ B\ .\]
  There is a relative correspondence
    \[ \wt{\DD}:= 2\Delta_{\wt{\Ss}}-  ({}^t \Gamma_{\wt{f}}) \circ \Gamma_{\wt{f}}\ \ \in A_{s-2}(\wt{\Ss}\times_B \wt{\Ss}) \]
   (here $s$ denotes the dimension of $\wt{\Ss}\times_B \wt{\Ss}$, $\Delta_{\wt{\Ss}}$ is the relative diagonal, $\Gamma_{\wt{f}}$ is the graph of $\wt{f}$, and relative correspondences over $B$ can be composed as in \cite{CH}, \cite{GHM}, \cite{NS}, \cite{DM}, \cite[8.1.2]{MNP} since $\wt{\Ss}$, $\wt{\MM}$ are smooth. At this point we grade the Chow group by dimension rather than codimension since $\wt{\Ss}\times_B \wt{\Ss}$ may be singular).
  For any $b\in B$, we have that $H^{0,2}(S_b)$ is a one--dimensional $\C$--vector space and
  \begin{equation}\label{ok}  ({f}_b)^\ast ({f}_b)_\ast = 2\hbox{id}\colon\ \ H^{0,2}({S}_b)\ \to\ H^{0,2}({S}_b)\ .\end{equation}
  
  We know that $H^{0,2}$ is a birational invariant for surfaces with rational singularities. (To see this, one notes that if $S$ is a surface with rational singularities and $\wt{S}\to S$ is a resolution of singularities the Leray spectral sequence implies $H^i(S,\OO_S)\to H^i(\wt{S},\OO_{\wt{S}})$ is an isomorphism, and so $\gr^0_F H^i(S,\C)\cong \gr^0_F H^i(\wt{S},\C)$ since rational singularities are Du Bois \cite[Theorem S]{Kov}). Hence, it follows from equality (\ref{ok}) that also
    
    \[ (\wt{f}_b)^\ast (\wt{f}_b)_\ast = 2\hbox{id}\colon\ \ H^{0,2}(\wt{S}_b)\ \to\ H^{0,2}(\wt{S}_b)\ \]
    
    Using the Lefschetz $(1,1)$ theorem on $\wt{S}_b$, this implies that for any $b\in B$, there exist a divisor $Y_b\subset \wt{S}_b$, and a cycle $\gamma_b\in A^2(\wt{S}_b\times \wt{S}_b)_{\QQ}$ supported on $Y_b\times Y_b$, such that
    \[  \wt{\DD}\vert_{\wt{S}_b\times \wt{S}_b}=\gamma_b\ \ \hbox{in}\ H^4(\wt{S}_b\times \wt{S}_b)\ ,\ \ \hbox{for\ all\ }b\in B\ .\]
    (Here, for any relative correspondence $\Gamma$, we use the notation $\Gamma\vert_{\wt{S}_b\times \wt{S}_b}$ to indicate the result
    of applying to $\Gamma$ the refined Gysin homomorphism \cite{F}  induced by $b\to B$.)
        
    Thanks to Voisin's ``spreading out'' result \cite[Proposition 2.7]{V0}, we can find a divisor $\YY\subset\wt{\Ss}$, and a cycle $\Gamma\in A_{s-2}(\wt{\Ss}\times_B \wt{\Ss})_{\QQ}$ supported on $\YY\times_B \YY$, with the property that the cycle
    \[  \wt{\DD}^\prime:= \wt{\DD}-\Gamma\ \ \in A_{s-2}(\wt{\Ss}\times_B \wt{\Ss})_{\QQ} \]
    has cohomologically trivial restriction to each fibre:
    \[  (\wt{\DD}^\prime)\vert_{\wt{S}_b\times \wt{S}_b}=0\ \ \hbox{in}\ H^4(\wt{S}_b\times \wt{S}_b)\ ,\ \ \hbox{for\ all\ }b\in B\ .\]
    
   After shrinking the base $B$ (i.e., after replacing $B$ by a Zariski open $B^\prime\subset B$), we may suppose that all the $S_b$ are smooth (Corollary \ref{genfibre}), and the morphisms
   $\VV\to B^\prime$,
    $\Ss\to B^\prime$ are smooth (so in particular, the fibre product $\Ss\times_{B^\prime} \Ss$ is smooth). Repeating the above procedure (or simply taking the push--forward of the restriction of $\wt{\DD}^\prime$), one finds
   a cycle
     \[ \DD^\prime \in A^{2}(\Ss\times_{B^\prime} \Ss)_{\QQ}\ .\]
    
     Note that there is a relation
     \[ \wt{\DD}^\prime\vert_{\wt{\Ss}\times_{B^\prime} \wt{\Ss}}= \phi^\ast (\DD^\prime) + \gamma\ \ \in A^{2}(\wt{\Ss}\times_{B^\prime} \wt{\Ss})_{\QQ}\ ,\]
     where 
     \[ \phi\colon \wt{\Ss}\times_{B^\prime}\wt{\Ss}\to \Ss\times_{B^\prime} \Ss\]
     is the birational morphism induced by the resolution morphism, and $\gamma$ is a cycle supported on
     \[ \Z\times_{B^\prime} \wt{\Ss}\ \cup\ \wt{\Ss}\times_{B^\prime} \Z\ ,\]
     for some divisor $\Z\subset \wt{\Ss}$. This is because the cycles $\wt{\DD}^\prime\vert_{\wt{\Ss}\times_{B^\prime} \wt{\Ss}}$ and $\phi^\ast (\DD^\prime)$ coincide outside of the exceptional locus of $\phi$, which is a divisor of the form     $ \Z\times_{B^\prime} \wt{\Ss}\ \cup\ \wt{\Ss}\times_{B^\prime} \Z $    
           (and more precisely: the extension of $\Z$ to the larger family $\wt{\Ss}\to B$ is such that $\Z_b\subset \wt{S}_b$ is a divisor for all $b\in B$, cf. Lemma \ref{desing}).
          
   Then, using a Leray spectral sequence argument as in \cite[Lemma 2.12]{V0} (and also as in \cite[Lemma 1.2]{V8}, where the set--up is exactly as here in the present proof), we know that after some further shrinking of the base $B^\prime$, there exists a cycle $c\in A^2(\PP\times \PP)_{\QQ}$ such that
    \[  \DD^{\prime\prime}:= \DD^\prime+(c\times B^\prime)\vert_{\Ss\times_{B^\prime} \Ss}=0\ \ \hbox{in}\ H^{4}(\Ss\times_{B^\prime} \Ss)\ .\]

      But then, since
     \[ A^{2}_{hom}(\Ss\times_{B^\prime} \Ss)_{\QQ}=0 \]
     by Proposition \ref{key} below, we have a rational equivalence
     \[   \DD^{\prime\prime}=0\ \ \hbox{in}\ A^{2}(\Ss\times_{B^\prime} \Ss)_{\QQ}\ .\]
     This implies that there is also a rational equivalence
     \[  \wt{\DD}^\prime\vert_{\wt{\Ss}\times_{B^\prime} \wt{\Ss}}=\phi^\ast\bigl((c\times B^\prime)\vert_{\Ss\times_{B^\prime} \Ss}\bigr)+\gamma\ \ \in A^{2}(\wt{\Ss}\times_{B^\prime} \wt{\Ss})_{\QQ}\ ,\]
     with $\gamma$ as above supported in codimension $1$.     
     
     Restricting to a general $b\in B$ (such that $b\in B^\prime$ and the divisor $\YY\subset \Ss$ restricts to a divisor $Y_b\subset S_b$), we now find a decomposition of the diagonal
     \[ \begin{split} 2 \Delta_{\wt{S}_b}= {}^t \Gamma_{\wt{f}_b}\circ \Gamma_{\wt{f}_b} &+\{\hbox{something\ supported\ on\ }Y_b\times Y_b\}\\ 
            & +\{\hbox{something\ supported\ on\ }\Z_b\times \wt{S}_b\cup \wt{S}_b\times \Z_b\}\\            &+ \{\hbox{something\ coming\ from\ }\PP\times\PP\}\ \ \ \ \hbox{in}\ A^2(\wt{S}_b\times \wt{S}_b)_{\QQ}\ .\\
           \end{split}\]
    Now, by considering the action of correspondences (and noting that only the first term acts on $A^2_{hom}(\wt{S}_b)_{\QQ}=A^2_{AJ}(\wt{S}_b)_{\QQ}$), this decomposition implies that
     \[    (\wt{f}_b)^\ast (\wt{f}_b)_\ast=2\hbox{id}\colon\ \ A^2_{hom}(\wt{S}_b)_{\QQ}\ \to\ A^2_{hom}(\wt{S}_b)_{\QQ}\ ,\ \ \hbox{for\ general\ }b\in B\ . \]
    This last equality (combined with the obvious fact that $(\wt{f}_b)_\ast (\wt{f}_b)^\ast$ is also twice the identity on Chow groups) proves 
    \[  A^2_{hom}(\wt{S}_b)_{\QQ}\cong A^2_{hom}(\wt{M}_b)_{\QQ}\]
     for the general $\wt{S}_b$. 
     
     To extend this statement to {\em all\/} $b\in B$, one considers the cycle 
           \[\wt{\DD}^{\prime}-\phi^\ast\bigl((c\times B)\vert_{\Ss\times_B \Ss}\bigr)-\bar{\gamma}\ \ \in A_{s-2}(\wt{\Ss}\times_{B} \wt{\Ss})_{\QQ}\ ,\]
           where $\bar{\gamma}$ denotes an extension of $\gamma$ that is still supported on an extension of the divisor $\Z$ over $B$ (and by abuse of language, we use the same symbol $\phi$ to indicate the induced morphism $\wt{\Ss}\times_B \wt{\Ss}\to \Ss\times_B \Ss$).
        For each $b$ in the open $B^\prime$, the restriction of this cycle to the fiber over $b$ is rationally trivial. Applying Lemma \ref{vois} below, it follows that the restriction of this cycle to {\em any\/} fibre is rationally trivial.  
          Next, given any $b_0\in B$, the moving lemma ensures that the divisor $\YY\subset \wt{\Ss}$ appearing in the construction may be chosen in general position with respect to $\wt{S}_{b_0}$; then, the above argument implies that
     \[  A^2_{hom} (\wt{S}_{b_0})_{\QQ}\cong A^2_{hom}(\wt{M}_{b_0})_{\QQ}  \ .\]

  \begin{lemma}\label{vois} Let $M\to B$ be a projective fibration, where $B$ is a smooth variety of dimension $r$. Let $\Gamma\in A_i(M)_{\QQ}$. The set of points $b\in B$ such that $\Gamma\vert_{M_b}=0$ in $A_{i-r}(M_b)_{\QQ}$ is  a countable union of closed algebraic subsets of $B$.
  \end{lemma}
  
  \begin{proof} Usually this is stated for $M$ smooth, for instance in \cite[Lemma 3.2]{Vo}. However, as the proof is just a Hilbert schemes argument, this still goes through for $M$ singular.
  \end{proof}

      Let us now wrap up the proof of Theorem \ref{main}: suppose $S$ is a Todorov surface with fundamental invariants $(1,10)$, and $P$ is a resolution of singularities of $S/\iota$. The canonical model of $S$ is an $S_b$ for some $b\in B$ (Corollary \ref{family}). After passing to a blow--up $\wt{S}$ of $S$, we get a diagram of surfaces
      \[ \begin{array}[c]{ccccc}
          \wt{S} &\to& S_b & \leftarrow & \wt{S}_b\\
          \downarrow && \downarrow && \downarrow\\
          P &\to& S_b/\iota_b &\leftarrow& \wt{M}_b\\
          \end{array}\]
         where horizontal arrows are birational morphisms (Lemma \ref{desing}), and surfaces in the left and right columns are smooth. We conclude using the commutative diagram
         \[  \begin{array}[c]{ccccc}
               A^2_{hom}(S)_{\QQ}&\xrightarrow{\cong}&A^2_{hom}(\wt{S})_{\QQ}  &\xrightarrow{\cong}& A^2_{hom}(\wt{S}_b)_{\QQ}\\
               &&\downarrow & &\ \ \downarrow{\cong}\\
              && A^2_{hom}(P)_{\QQ}&\xrightarrow{\cong}& A^2_{hom}(\wt{M}_b)_{\QQ}\\
               \end{array}\]
      (here horizontal arrows are isomorphisms because $A^2_{hom}$ is a birational invariant for smooth surfaces, and the right vertical arrow is an isomorphism as we have shown above).

The above argument relies on the following key result, the proof of which is postponed to the next section:

\begin{proposition}\label{key} Let 
  $  \VV \to B$
  be the family of weighted complete intersection surfaces, and let $\Ss\to B$ be the family of Todorov surfaces as in Corollary \ref{family}. Suppose $B$ is small enough for the morphism $\VV\to B$ to be smooth. Then
  \[ A^{2}_{hom}(\Ss\times_B \Ss)_{\QQ} =  A^2_{hom}(\VV\times_B \VV)_{\QQ} =0\ .\]
  \end{proposition}
  
\end{proof}

We now state a few corollaries of Theorem \ref{main}:

\begin{corollary} Let $S$ be a Todorov surface with $K^2_S=2$ and $\pi_1(S)=\ZZ/2\ZZ$. Then the generalized Hodge conjecture is true for the sub--Hodge structure
  \[ \wedge^2 H^2(S)\ \subset\ H^4(S\times S)\ .\]
  The Hodge conjecture is true for $(2,2)$--classes in $\wedge^2 H^2(S)$.
  \end{corollary}
  
  \begin{proof} As already noted in \cite{V9}, this follows from Corollary \ref{true} using the Bloch--Srinivas method \cite{BS}.
  \end{proof}
  
\begin{corollary}\label{intersection} Let $S$ be a Todorov surface with $K^2_S=2$ and $\pi_1(S)=\ZZ/2\ZZ$. Let $\iota$ be the involution such that $S/\iota$ is birational to a $K3$ surface, and let
$A^1(S)^\iota$ denote the $\iota$--invariant part of $A^1(S)$. Then
  \[ \ima \Bigl( A^1S\otimes A^1(S)^\iota \xrightarrow{\cdot} A^2S \Bigr) \]
  has dimension $1$.
  \end{corollary}
  
  \begin{proof} 
  In view of Rojtman's theorem, it suffices to prove the statement with rational coefficients. Let $p\colon S\to S/\iota$ denote the projection. Since $S/\iota$ is birational to a $K3$ surface (in other words, $S/\iota$ is a ``$K3$ surface with rational double points'', in the language of \cite{Mor3}), there is a distinguished $0$--cycle $e\in A^2(S/\iota)$, with the property that
    \[ \ima \Bigl( A^1(S/\iota)_{\QQ}\otimes A^1(S/\iota)_{\QQ} \xrightarrow{\cdot} A^2(S/\iota)_{\QQ}\Bigr)= \QQ\cdot e \]
  \cite{BV}.  
  
  Now given two divisors $D\in A^1(S)_{\QQ}$ and $D^\prime\in A^1(S)_{\QQ}^\iota$, we can write $D^\prime= p^\ast(F^\prime)$ for some $F^\prime\in A^1(S/\iota)_{\QQ}$. Using the projection formula, we find that
    \[ p_\ast ( D\cdot D^\prime)=p_\ast \bigl(D\cdot p^\ast(F^\prime)\bigr)= p_\ast(D)\cdot F^\prime=\hbox{deg}(p_\ast(D)\cdot F^\prime) e\ \ \hbox{in}\ A^2(S/\iota)_{\QQ}\ .\]
    Let $e_S\in A^2(S)$ be any $0$--cycle mapping to $e\in A^2(S/\iota)$. Then (as $A^2(S)_{\QQ}\to A^2(S/\iota)_{\QQ}$ is an isomorphism by Theorem \ref{main}), we have
    \[  D\cdot D^\prime = \hbox{deg}(p_\ast(D)\cdot F^\prime) e_S\ \ \hbox{in}\ A^2 (S)_{\QQ}\ ;\]
    that is, $e_S$ can be considered a ``distinguished $0$--cycle'' for the intersection on $S$.
    \end{proof}
  
 \begin{remark} An equivalent formulation of Corollary \ref{intersection} is as follows: for $S$ a surface as in Corollary \ref{intersection}, there exists $e_S$ such that for all divisors $D_1, D_2\in A^1(S)$, we have
   \[ D_1\cdot \bigl(D_2+\iota_\ast(D_2)\bigr) = c e_S\ \ \hbox{in}\ A^2 S\ ,\]
   for some $c\in\ZZ$.
   \end{remark}

\begin{remark} A result similar to (but stronger than) Theorem \ref{main} is proven by Voisin for $K3$ surfaces. Voisin proves \cite{V11} that if $X$ is any $K3$ surface, and $\iota$ is a symplectic involution of $X$ then $A^2(X)_{}=A^2(X)^\iota_{}$. Our result Theorem \ref{main} is weaker than this, in the sense that we can not prove anything for an arbitrary symplectic involution on a surface $S$ as in Theorem \ref{main}; our proof only works if the involution extends to the whole family of Todorov surfaces with the given invariants.
\end{remark}

\begin{remark} Below we will prove (Theorem \ref{main2}) the motivic version of Theorem \ref{main} that was stated in the introduction. This motivic version is not necessary for the proof of Corollary \ref{true} (for which the statement of Theorem \ref{main} suffices), but it may have some independent interest.
\end{remark}

\section{Trivial Chow groups}
\label{seckey}

This section contains the proof of Proposition \ref{key}, which was a key result used in the preceding section. We rely on work of Totaro \cite{T}, which is recalled in subsection \ref{totsection}. Subsection \ref{proofsection} proves Proposition \ref{key}, by considering an appropriate stratification of $\wt{\PP\times\PP}$. Things work out fine, because ``everything is linear'' (i.e., all the strata, and all their intersections, look like affine spaces).

\subsection{Weak and strong property}
\label{totsection}

  \begin{definition}[Totaro \cite{T}] For any (not necessarily smooth) quasi--projective variety $X$, let $A_i(X,j)$ denote Bloch's higher Chow groups (these groups are sometimes written $A^{n-i}(X,j)$ or $CH^{n-i}(X,j)$, where $n=\dim X$). As explained in \cite[Section 4]{T}, the relation with algebraic $K$--theory ensures there are functorial cycle class maps
    \[  A_i(X,j)_{\QQ}\ \to\ \gr^W_{-2i} H_{2i+j}(X)\ ,\]
    compatible with long exact sequences (here $W$ denotes Deligne's weight filtration on Borel--Moore homology \cite{PS}).
    
    We say that $X$ has the {\em weak property\/} if the cycle class maps induce isomorphisms
    \[   A_i(X)_{\QQ}\ \xrightarrow{\cong}\ W_{-2i} H_{2i}(X)\]
    for all $i$.
    
    We say that $X$ has the {\em strong property\/} if $X$ has the weak property, and, in addition, the cycle class maps induce surjections
      \[  A_i(X,1)_{\QQ}\ \twoheadrightarrow\ \gr^W_{-2i} H_{2i+1}(X) \]
      for all $i$.
   \end{definition}

\begin{lemma}\label{local} Let $X$ be a quasi--projective variety, and $Y\subset X$ a closed subvariety with complement $U=X\setminus Y$. If $Y$ and $U$ have the strong  property, then so does $X$.
\end{lemma}

\begin{proof} This is the same argument as \cite[Lemma 7]{T}, which is a slightly different statement. As in loc. cit., using the localization property of higher Chow groups \cite{B3}, \cite{Lev}, one finds a commutative diagram with exact rows
  \[ \begin{array}[c]{cccccccc}
                                                       A_{i}(U,1)_{\QQ}&\to& A_i(Y)_{\QQ}  &\to&      
                                                                                         A_i(X)_{\QQ} &\to& A_i(U)_{\QQ} & \to 0\\
                                            \downarrow && \downarrow && \downarrow && \downarrow & \\
                                                                       \gr^W_{-2i}H_{2i+1}(U) &\to& \gr^W_{-2i} H_{2i}(Y) &\to& \gr^W_{-2i} H_{2i}(X) &\to& \gr^W_{-2i} H_{2i}(U)& \to 0 \\
 \end{array} \]    
   A diagram chase reveals that under the assumptions of the lemma, the one but last vertical arrow is an isomorphism.
   
   Continuing these long exact sequences to the left, there is a commutative diagram with exact rows
    \[ \begin{array}[c]{cccccccc}


      A_{i}(Y,1)_{\QQ}&\to&  A_{i}(X,1)_{\QQ}&\to&  A_{i}(U,1)_{\QQ}&\to& A_i(Y)_{\QQ} &\to\\  
       \downarrow && \downarrow && \downarrow && \ \ \ \downarrow{\cong} & \\
      \gr^W_{-2i} H_{2i+1}(Y)&\to&      \gr^W_{-2i} H_{2i+1}(X)&\to&      \gr^W_{-2i} H_{2i+1}(U) &\to& \gr^W_{-2i} H_{2i}(Y) &\to\\
       \end{array}\]
      
      Doing another diagram chase, one learns that the second vertical arrow is a surjection.       
          \end{proof}

  \begin{corollary}\label{linear} Let $X$ be a quasi--projective variety that admits a stratification such that each stratum is of the form $\A^k\setminus L$, where $L$ is a finite union of linearly embedded affine subspaces. Then $X$ has the strong property.
  \end{corollary}
  
  \begin{proof} Affine space has the strong property (this is the homotopy invariance for higher Chow groups). The subvariety $L$ has the weak property. Doing a diagram chase as in lemma \ref{local} (or directly applying \cite[Lemma 6]{T}), it follows that the variety
  $\A^k\setminus L$  has the strong property. The corollary now follows from lemma \ref{local}.
  \end{proof}
  
 \begin{lemma}\label{projbundle} Let $X$ be a quasi--projective variety with the strong property. Let $Y\to X$ be a projective bundle. Then $Y$ has the strong property.
 \end{lemma}
 
 \begin{proof} This follows from the projective bundle formula for higher Chow groups \cite{B2}.
  \end{proof}

 \subsection{Proof of Proposition \ref{key}}
\label{proofsection}

We now proceed to prove the key proposition:

\begin{nonumberingp}[(=Proposition \ref{key})] Let 
  $  \VV \to B$
  be the family of weighted complete intersection surfaces, and let $\Ss\to B$ be the family of Todorov surfaces as in Corollary \ref{family}. Let $B^\prime\subset B$ be an open such that the induced morphism $\VV^\prime\to B^\prime$ is smooth. Let $\Ss^\prime\to B^\prime$ denote the restriction of $\Ss$ to $B^\prime$. Then
  \[ \begin{split} 
   &A^2_{hom}(\Ss^\prime\times_{B^\prime} \Ss^\prime)_{\QQ}=0\ ,\\
     &A^2_{hom}(\VV^\prime\times_{B^\prime} \VV^\prime)_{\QQ}=0\ .\\
     \end{split}\]
  \end{nonumberingp}
  
\begin{proof}
Since there is a finite surjective morphism
  \[ \VV^\prime\times_{B^\prime} \VV^\prime\ \to\ \Ss^\prime\times_{B^\prime} \Ss^\prime\ , \]
  the first statement follows from the second. 
  To prove the second statement, we will actually prove the following:
  
 \begin{proposition}\label{key2} Let $\VV^\prime\to B^\prime$ be as in Proposition \ref{key}, and let
   \[  \wt{\VV^\prime\times_{B^\prime} \VV^\prime}\ \to\  \VV^\prime\times_{B^\prime} \VV^\prime\]
   be the blow--up along the relative diagonal. There exists a projective variety $M$ with
   \[  A_\ast^{hom}(M)_{\QQ}=0\ ,\]
   and such that $M$ contains $\wt{\VV^\prime\times_{B^\prime} \VV^\prime}$ as a Zariski open.   
    \end{proposition} 

It is easily seen that Proposition \ref{key2} implies Proposition \ref{key}: indeed, set
   \[ U:=\wt{\VV^\prime\times_{B^\prime} \VV^\prime}\ ,\ \ D:=M\setminus U\ ,\]
   and let $m:=\dim M$.
   Suppose $a\in A^2_{hom}(U)_{\QQ}$. Then there exists $\bar{a}\in A_{m-2}(M)_{\QQ}$ restricting to $a$, and such that the class
     \[ [\bar{a}]\ \ \in H_{2m-4}(M)\]
     maps to $0$ in $H^4U$. Using a resolution of singularities of $M$, one finds that the homology class $[\bar{a}]$ comes from a Hodge class $\beta\in H^2(\wt{D})$ (where $\wt{D}\to D$ is
     a resolution of singularities of the boundary divisor $D$). The Lefschetz $(1,1)$ theorem ensures that the class $\beta$ is algebraic, say $\beta=[b]$ for some $b\in A^1(\wt{D})_{\QQ}$. Now
     \[ \bar{a}^\prime:=\bar{a}-i_\ast(b)\ \ \]
     is a class in $A_{m-2}^{hom}(M)_{\QQ}=0$ restricting to $a$, and hence $a=0$. This clearly implies that also
       \[A^{2}_{hom}(\VV^\prime\times_{B^\prime} \VV^\prime)_{\QQ}=0\] (if $\phi\colon U\to \VV^\prime\times_{B^\prime} \VV^\prime$ denotes the blow--up, we have that $\phi_\ast\phi^\ast=\hbox{id}$ on $A^2_{hom}(\VV^\prime\times_{B^\prime} \VV^\prime)_{\QQ}$).
       
       (NB: a stronger statement
         \[  A^i_{hom}(\VV^\prime\times_{B^\prime} \VV^\prime)_{\QQ}=0\ \ \ \hbox{for\ all\ }i\ \]
        is likely to be true, cf. Remark \ref{vsc}.)

   We now proceed to prove Proposition \ref{key2}; this is a slight modification of an argument of Voisin (\cite[Proposition 2.13]{V0} and \cite[Lemma 1.3]{V1}, also explained in \cite[Section 4.3]{Vo}).
   Let
   \[ \bar{B}\ \supset\ B\]
   denote the projective closure of $B$ (so $\bar{B}$ is a product of two projective spaces $\PP^r\times\PP^r$). Let
   \[ \wt{\PP\times\PP}\ \to\ \PP\times\PP\ \]
   be the blow--up along the diagonal. Points of $\wt{\PP\times\PP}$ correspond to the data of $(x,y,z)$, where $x,y$ are points of $\PP$ and $z\subset X$ is a length $2$ zero--dimensional subscheme with associated cycle $x+y$. Consider now the variety
   \[ M:= \Bigl\{ \bigl((F_b,G_b),(x,y,z)\bigr)\ \vert\ F_b\vert_z=G_b\vert_z=0\Bigr\}\ \ \subset\ \bar{B}\times \wt{\PP\times\PP}\ .\]
   Clearly $M$ contains $\wt{\VV^\prime\times_{B^\prime} \VV^\prime}$ as a Zariski open. We now proceed to show that $M$ has trivial Chow groups. Note that the fibre of the projection
   \[  \pi\colon\ \   M\ \to\ \wt{\PP\times\PP}\]
   over a point $(x,y,z)\in\wt{\PP\times\PP}$ is
     \[  \bigl\{ b\in \bar{B}\ \vert\ F_b\vert_z=G_b\vert_z=0\bigr\}\ \ \subset\ \bar{B}\ ,\]
     which is of the form
     \[    \PP^s\times \PP^t\ \subset\ \PP^r\times\PP^r=\bar{B}   \ .\]
     
    The strategy of this proof will be to stratify $\wt{\PP\times\PP}$ such that over each stratum, the morphism $\pi$ has constant dimension.

   It follows from Lemma \ref{one} that, with two exceptions, every point imposes one condition on the polynomials $F_b, G_b$, i.e. for all 
     \[ x\in\PP\setminus ( [0:0:0:1:0] \cup [0:0:0:0:1])\ ,\]
    we have that
   \[  \bigl\{ b\in\bar{B}\ \vert\ V_b\ni x\bigr\}\cong \PP^{r-1}\times\PP^{r-1}\ \ \subset \PP^r\times\PP^r=\bar{B}\ .\]

   It remains to analyze what happens when we impose two points. 
   Let's define the locus
     \[ \begin{split} \bar{Q}:=f^{-1}\Bigl( ([0&:0:0:1:0]\cup [0:0:0:0:1])\times\PP\\ \ &\cup \ \PP\times ([0:0:0:1:0]\cup [0:0:0:0:1])\Bigr)\ \ \subset \wt{\PP\times\PP}\\
         \end{split} \]
     (where $f\colon\wt{\PP\times\PP}\to\PP\times\PP$ is the blow--up of the diagonal).
        
  We leave aside (for later consideration) $\bar{Q}$ and $E$, that is we write       
      \[ P:=\wt{\PP\times\PP}\setminus (E\cup \bar{Q})\ \]
     (so $P_{}$ is isomorphic to an open in $(\PP\times\PP)\setminus \Delta$).
   
   We now proceed to stratify $P_{}$, as follows:
     First, we define ``partial diagonals''
     \[  \begin{split} \Delta_{3,\pm,\pm}:= \Bigl\{ (p,p^\prime)\in\PP\times\PP\ \vert\  \exists \lambda\in\C^\ast \hbox{\ such\ that\ } p_1=\lambda p_1^\prime &\hbox{\ and\ } p_2=\lambda p_2^\prime\\
       &\hbox{\ and\ }p_0=
            \pm\lambda p_0^\prime\hbox{\ and\ } p_4=\pm \lambda^2 p_4^\prime\Bigr\}\ ,\\
             \Delta_{4,\pm,\pm}:= \Bigl\{ (p,p^\prime)\in\PP\times\PP\ \vert\  \exists \lambda\in\C^\ast \hbox{\ such\ that\ } p_1=\lambda p_1^\prime &\hbox{\ and\ } p_2=\lambda p_2^\prime\\ 
             &\hbox{\ and\ } p_0=\pm\lambda 
                  p_0^\prime   \hbox{\ and\ } p_3=\pm \lambda^2 p_3^\prime\Bigr\}\ \\
             \end{split}\]
      (here we suppose a point $p\in\PP$  has coordinates $p=[p_0:p_1:p_2:p_3:p_4]$). 
      
      (Just to fix ideas: we have for example that $\Delta_{3,+,+}\cap \Delta_{4,+,+}$ is the diagonal of $\PP$.)
           
    We define closed subvarieties $P_{1,j}\subset P_{}$ as follows:
          \[  \begin{split} P_{1,1}&:= (\Delta_{3,+,+})\cap P_{}\ ,\\
                              P_{1,2}&:= (\Delta_{3,+,-})\cap P\ ,\\
                              P_{1,3}&:= (\Delta_{3,-,+})\cap P\ ,\\   
                              P_{1,4}&:= (\Delta_{3,-,-})\cap P\ ,\\   
                              P_{1,5}&:= (\Delta_{4,+,+})\cap P\ ,\\
                              P_{1,6}&:= (\Delta_{4,+,-})\cap P\ ,\\
                              P_{1,7}&:= (\Delta_{4,-,+})\cap P\ ,\\   
                              P_{1,8}&:= (\Delta_{4,-,-})\cap P\ ,\\   
                            \end{split}\]
             and an open subvariety
             \[  P_{}^0:=  P\setminus (\bigcup_j P_{1,j})\]
             (that is, $P_{}^0$ is the complement in $(\PP\times\PP)\setminus f(\bar{Q})$ of the union of the various partial diagonals $\Delta_{3,\pm,\pm}, \Delta_{4,\pm,\pm}$).     
                                  
    We next define closed subvarieties 
      \[  \begin{split} P_{2,1}&:= P_{1,1}\cap   P_{1,5}\ ,\\    
                           P_{2,2}&:= P_{1,1}\cap P_{1,6}\ ,\\
                           P_{2,3}&:= P_{1,1}\cap P_{1,7}\ ,\\
                           P_{2,4}&:= P_{1,1}\cap P_{1,8}\ ,\\
                           P_{2,5}&:= P_{1,2}\cap P_{1,5}\ ,\\
                           P_{2,6}&:= P_{1,2}\cap P_{1,6}\ ,\\
                            \ldots\ \\
                           P_{2,16}&:= P_{1,4}\cap P_{1,8}\ .\\
                           \end{split}\]
                    There are open subvarieties $P_{1,j}^0\subset P_{1,j}$ defined as
                    \[ P_{1,j}^0:= P_{1,j}\setminus (\bigcup_{P_{2,k}\subset P_{1,j}}   P_{2,k})\ .\]

           The upshot is that we have a stratification 
              \[      P_2\subset P_1\subset P=\wt{\PP\times\PP}\setminus E\ ,\]
              where $P_i:=\cup_{j} P_{i,j}$, such that at each step 
              \[    P_i\setminus P_{i+1} = \bigcup_j  P_{i,j}^0 \ .\]
              (Here, by convention, we write $P=P_{0,0}$ and $P^0=P_{0,0}^0$.) 
 
 We now return to the morphism
   \[ \pi\colon\ \ M\ \to\ \wt{\PP\times\PP}\]
   defined above; by construction, each fibre $F$ of $\pi$ is of type 
   \[ F\cong \PP^s\times \PP^t\subset \PP^r\times \PP^r=\bar{B}\ .\]
   Let
   \[ M_2:=\pi^{-1}(P_2)\ ,\ \ M_{1,j}^0:=\pi^{-1}(P_{1,j}^0)\ ,\ \ M_0^0:=\pi^{-1}(P_0^0)\ ;\]
          we thus obtain a stratification of $M_0:=M\setminus \pi^{-1}(E\cup Q)$.
   The point of doing this, is that over each stratum the morphism $\pi$ is of constant dimension:
              
 \begin{lemma}\label{constantdim} Over each stratum of $M_0\to P_0$, the morphism $\pi$ restricts to
    a fibration with fibres $\PP^s\times \PP^t$.
 
 More precisely: a fibre $F=\pi^{-1}(p)$ is
   \[ F\cong \begin{cases}  \PP^{r-1}\times \PP^{r-1}&\hbox{if}\ p\in P_2\ ;\\
                                         \PP^{r-2}\times \PP^{r-1}&\hbox{if}\  p\in P_{1,j}^0\ \hbox{with}\ 1\le j\le 4\ ;\\
                                          \PP^{r-1}\times\PP^{r-2}&\hbox{if}\  p\in P_{1,j}^0\ \hbox{with}\ 5\le j\le 8\ ;\\    
                                        \PP^{r-2}\times\PP^{r-2} &\hbox{if}\ p\in P^0\ .\\
                                        \end{cases}\]
                                           \end{lemma}

\begin{proof} It is readily seen that a point $p$ which lies on a partial diagonal $\Delta_{3,\pm,\pm}$ imposes at most $1$ condition on the
polynomials $G_b$ of Corollary \ref{family}. Combined with Lemma \ref{one}, this observation yields that points on
  \[ \bigcup_{}\Delta_{3,\pm,\pm}\setminus \bigl( (Q\cup E)\cap (\bigcup_{}\Delta_{3,\pm,\pm})\bigr)\]
  impose exactly $1$ condition on polynomials $G_b$ as in corollary \ref{family}.
  On the other hand, given any point
  \[  p=(q,q^\prime)\ \ \in \PP\times\PP\setminus ( Q\cup \bigcup_{}\Delta_{3,\pm,\pm})\ ,\]
  it is readily seen there exists $G_b$ as in Corollary \ref{family} separating the points $q,q^\prime$, i.e. $p$ imposes $2$ independent conditions on the $G_b$.
  
  The same observation can be made concerning the partial diagonals $\Delta_{4,\pm,\pm}$: a point on a $\Delta_{4,\pm,\pm}$ and not on $Q\cup E$ imposes
  exactly $1$ condition on the polynomials $F_b$, while points outside of the $\Delta_{4,\pm,\pm}\cup Q\cup E$ impose $2$ independent conditions on the $F_b$.
  
  Combining these two observations proves the lemma. 
\end{proof}

\begin{lemma}\label{strata} Each of the strata 
    \[ P_{}^0,\  
    \bigcup_{1\le j \le4} P_{1,j}^0,\ 
   \bigcup_{5\le j\le 8} P_{1,j}^0
      ,\ P_2 \] 
can be written as a disjoint union of varieties of type $\A^k\setminus L$, where $L$ is a finite union of linearly embedded affine spaces. 

\end{lemma}
   
\begin{proof} First, consider $P^0=P_{0,0}^0$. By definition, this is nothing but
  \[ \PP\times\PP\setminus (Q\cup \bigcup_{}   \Delta_{3,\pm,\pm}\cup \Delta_{4,\pm,\pm})\ .\]
  Let $U\subset\PP$ be the open subset $(w_0\not=0)$. Then $P^0\cap (U\times U)$ is isomorphic to $\A^8$ minus $8$ copies of $\A^5$ that are linearly embedded.
  The intersection 
  \[ P^0\cap \bigl( (w_0=0)\times (w_0\not=0)\bigr)\] 
  can be identified with $\PP(1,1,2,2)\times \A^4$. It remains to consider 
  \[  P^0\cap \bigl( (w_0=0)\times (w_0=0)\bigr)\ ;\]
  the argument is similar (restricting to the open $(x_1\not=0)$ we find again a stratum of the requisite type).
  
  Next, consider $P_{1,j}^0$. The intersection
    \[  P_{1,j}^0\cap ( U\times U) \]
    is isomorphic to $\A^5$ minus $4$ copies of $\A^4$ that are linearly embedded. Since all intersections are linear subspaces, the assertion for the unions
    \[  \bigcup_{1\le j \le4} P_{1,j}^0,\ 
   \bigcup_{5\le j\le 8} P_{1,j}^0    \]
   follows from this.
   As for $P_2$, this is similar: the intersection
   \[  P_2\cap (U\times U) \]
   is a union of copies of $\A^4$ that are linearly embedded in $\A^8$; in particular the intersections of the irreducible components are again affine spaces. 
\end{proof}   
         
\begin{corollary}\label{strong} The open $M_0:=M\setminus \pi^{-1}(E\cup Q)$ has the strong property.
\end{corollary}

\begin{proof} It follows from Lemma \ref{strata}, combined with Lemma \ref{linear}, that $P_2$ has the strong property. Since $M_2=\pi^{-1}(P_2)$ is a fibration over $P_2$ with fibres products of projective spaces, it follows that $M_2$ has the strong property (Lemma \ref{projbundle}). 

The strata 
  \[ \bigcup_{1\le j\le 4} M_{1,j}^0, \ \ \bigcup_{5\le j\le 8} M_{1,j}^0  \]
   are fibrations over 
   \[  \bigcup_{1\le j\le 4}P_{1,j}^0\ ,\hbox{resp.}\ \  \bigcup_{5\le j\le 8}P_{1,j}^0\ ,\]
   with fibre a product of projective spaces (Lemma \ref{constantdim}). The base has the strong property (Lemmas \ref{strata} and \ref{linear}), hence these strata of $M$ have the strong property (Lemma \ref{projbundle}). Using Lemma \ref{local}, it follows that 
  $M_1$ has the strong property. One similarly finds that $M_0^0=\pi^{-1}(P_{0,0}^0)$ has the strong property, and hence (applying Lemma \ref{local} again) that
  $M^0$ has the strong property.
\end{proof}   

We now return to the closed subset $\bar{Q}$ that we left aside; more precisely, we consider the locally closed subset outside of the exceptional divisor
  \[ Q:= \bar{Q}\setminus (\bar{Q}\cap E)\ \ \subset \wt{\PP\times\PP}\setminus E\ \ \ (\cong \PP\times\PP\setminus \Delta)\ .\]
  
  We proceed to stratify $Q$. We define closed subvarieties:
      \[  \begin{split} Q_{1,1}&:= [0:0:0:1:0]\times\PP\setminus ([0:0:0:1:0]\times[0:0:0:1:0])\ ,\\
                          Q_{1,2}&:=[0:0:0:0:1]\times\PP\setminus ([0:0:0:0:1]\times[0:0:0:0:1])\ ,\\
                          Q_{1,3}&:=\PP\times [0:0:0:1:0]\setminus ([0:0:0:1:0]\times[0:0:0:1:0])\ ,\\                          
                            Q_{1,4}&:=     \PP\times[0:0:0:0:1]\setminus ([0:0:0:0:1]\times[0:0:0:0:1])\ ,\\
                            Q_{2,1}&:=Q_{1,1}\cap (\Delta_{4,+,+})\ ,\\
                             Q_{2,2}&:=Q_{1,2}\cap (\Delta_{3,+,+})\ ,\\
                            Q_{2,3}&:=Q_{1,3}\cap (\Delta_{4,+,+})\ ,\\
                            Q_{2,4}&:=Q_{1,4}\cap (\Delta_{3,+,+})\ ,\\
                             Q_{2,5}&:=Q_{1,1}\cap Q_{1,4}\ ,\\
                            Q_{2,6}&:=Q_{1,2}\cap Q_{1,3}\ .\\
                            \end{split}\]
               We also define open subvarieties
               \[   Q^0_{1,j}:= Q_{1,j}\setminus (\bigcup_{Q_{2,k}\subset Q_{1,j}} Q_{2,k})\ ,\ \ j=1,\ldots,4\ .\]

  \begin{lemma} The varieties $Q^0_{1,j}$ and $Q_{2,k}$ have the strong property.
  \end{lemma} 
  
  \begin{proof} This is readily deduced from Lemma \ref{local}. Each $Q_{1,j} (j=1,\ldots,4)$ is a copy of $\PP$ with a point taken out; each $Q_{2,k}$ with $1\le k\le 4$ 
  is a copy of $\PP^1$ with a point taken out; $Q_{2,5}$ and $Q_{2,6}$ are just points; each $Q_{1,j}^0 (j=1,\ldots,4)$ is isomorphic to $\PP$ with a $\PP^1$ and a point taken out.
  \end{proof}
  
  Consider now the restriction of the morphism $\pi$ to
    \[  M_{Q}:= \pi^{-1}(Q)\ \to\ Q\ ,\]
and to the various strata of $M_Q$ defined by the stratification of $Q$:
  \[ \begin{split} M_{Q_{1,j}^0}&:= \pi^{-1}(Q_{1,j}^0)\ ,\ \ j=1,\ldots,4\ ,\\
                         M_{Q_{2,k}}&:= \pi^{-1}( Q_{2,k})\ ,\ \ k=1,\ldots,6\ .\\
                      \end{split}\]   
  
  \begin{lemma} The restriction of $\pi$ has constant dimension on each $M_{Q^0_{1,j}}$, and on each $M_{Q_{2,k}}$.
    \end{lemma}
    
  \begin{proof} Consider a point $q$ on the stratum $Q_{1,1}^0$ (the argument for the other $Q_{1,j}^0$ is only notationally different). 
  We have
    \[ q=\bigl( [0:0:0:1:0], p\bigr)\ ,\]
    for some $p\in\PP$.
    Clearly {\em all\/} polynomials $G_b$ as in Corollary \ref{family} pass through $[0:0:0:1:0]$. Since we are outside of the partial diagonals $\Delta_{3,\pm,\pm}$ and $\Delta_{4,\pm,\pm}$, the point $q$ imposes one condition on the polynomials $G_b$, and two conditions on the polynomials $F_b$ (cf. the proof of Lemma \ref{constantdim}). It follows that the fibre is
    \[ \pi^{-1}(q)\cong \PP^{r-2}\times \PP^{r-1}\ .\]
    The argument for the strata $Q_{2,k}, 1\le k\le 4$ is similar; consider for example a point $q\in Q_{2,1}$. Such a $q$ imposes one condition on the $G_b$, and one condition on the $F_b$ and so
    \[ \pi^{-1}(q)\cong \PP^{r-1}\times \PP^{r-1}\ .\]
    The points 
    \[  Q_{2,5}=\bigl( [0:0:0:1:0], [0:0:0:0:1]\bigr)\ \]
    and 
    \[ Q_{2,6}= \bigl( [0:0:0:0:1], [0:0:0:1:0]\bigr)\ \]  
   likewise impose one condition on the $F_b$ and one condition on the $G_b$, and hence
   \[ \pi^{-1}(Q_{2,k})\cong \PP^{r-1}\times \PP^{r-1}\ \ \ , k=5,6\ .\]              
      \end{proof}

   \begin{corollary}\label{strong3} The variety $M_Q$ has the strong property.
    \end{corollary} 
    
    \begin{proof} This is immediate from the above two lemmas, using Lemma \ref{local}.
    \end{proof}

  It remains to stratify the exceptional divisor $E$ of the blow--up
    \[ f\colon\ \  \wt{\PP\times\PP}\ \to\ \PP\times\PP\ ,\] 
    in a similar way. A point on $E$ is given by the data
    \[   \bigl\{    \bigl(x,t\bigr)\in \PP\times \PP^3\bigr\}\ ,\]
    where $(x,x)$ is a point on the diagonal, and $L_t$ is a line in $\PP\times\PP$ passing through $(x,x)$ and not contained in the diagonal. 
    Consider the following loci:
    \[ \begin{split} E_{1,1}&:=f^{-1}([0:0:0:1:0])\ ,\\
                           E_{1,2}&:=f^{-1}([0:0:0:0:1])\ ,\\    
                           E_{1,3}&:= E\cap \bar{\Delta}_{3,+,+}\ ,\\
                            E_{1,4}&:= E\cap \bar{\Delta}_{4,+,+}\ ,\\
                            \end{split}\]
                 where $\bar{\Delta}_{j,+,+}$ denotes the strict transform of $\Delta_{j,+,+}$. 
                 
                 (That is, $E_{1,3}$ and $E_{1,4}$ parametrize lines not contained in the diagonal that remain inside $\Delta_{3,+,+}$ resp. $\Delta_{4,+,+}$.) 
                     
                 We define $E^0$ as the open complement:
                 \[ E^0:= E\setminus (\cup_j E_{1,j})\ .\]
                 We also define points
                 \[   \begin{split}  E_{2,1}&:=E_{1,1}\cap E_{1,3}\ ,\\
                                           E_{2,2}&:= E_{1,1}\cap E_{1,4}\ ,\\
                                           E_{2,3}&:=E_{1,2}\cap E_{1,3}\ ,\\
                                              E_{2,4}&:=E_{1,2}\cap E_{1,4}\ ,\\
                                       \end{split}\]   
                 and
                  locally closed subvarieties
                 \[  \begin{split} E_{1,1}^0&:= E_{1,1}\setminus (E_{2,1}\cup E_{2,2})\ ,\\
                                          E_{1,2}^0&:= E_{1,2}\setminus (E_{2,3}\cup E_{2,4})\ ,\\
                                          E_{1,3}^0&:= E_{1,3}\setminus (E_{2,1}\cup E_{2,3})\ ,\\
                                          E_{1,4}^0&:= E_{1,4}\setminus (E_{2,2}\cup E_{2,4} \ .\\
                                    \end{split}\]

    \begin{lemma} The varieties $E^0$, $E_{1,j}^0$ and $E_{2,j}$ have the strong property.
    \end{lemma}
    
    \begin{proof} The $E_{2,j}$ are just points. The varieties $E_{1,1}^0$ and $E_{1,2}^0$ are isomorphic to $\PP^3$ minus two points, so this is again obvious. The varieties $E_{1,3}^0, E_{1,4}^0$ are isomorphic to the diagonal of $\PP$ minus two points. Applying Lemma \ref{local}, it follows that
      \[ \bigcup_j E_{1,j}\]
      has the strong property.
      As for $E^0$: clearly $E$ has the strong property; we take out $\cup_j E_{1,j}$ which has the strong property: the result is something with the strong property (\cite[Lemma 6]{T}).  
    \end{proof}
   
 We now return to the morphism $\pi\colon M\to\wt{\PP\times\PP}$ defined above. The pre--image
   \[ M_E:= \pi^{-1}(E)\]
   admits a stratification as a disjoint union
   \[  M_E= M_{E^0}\cup \bigcup_{1\le j\le 4} M_{E_{1,j}^0}\cup \bigcup_{1\le k\le 4} M_{E_{2,k}}  \ ,\]
   where $M_{E^0}, M_{E^0_{1,j}}, M_{E_{2,k}}$ are defined as $\pi^{-1}(E_0)$ resp. $\pi^{-1}(E^0_{1,j})$ resp. $\pi^{-1}(E_{2,k})$.
   
   On each stratum, the morphism $\pi$ is of constant dimension:
    
   \begin{lemma} The fibre $F=\pi^{-1}(p)$ of  $\pi\colon M\to\wt{\PP\times\PP}$ is
     \[  F\cong \begin{cases}        \PP^{r-2}\times\PP^{r}&\hbox{if\ }p= E_{2,1}\ ;\\     
                                                        \PP^{r-1}\times\PP^{r-1}&\hbox{if\ }p= E_{2,2}\ ;\\
                                                         \PP^{r-1}\times\PP^{r-1}&\hbox{if\ }p= E_{2,3}\ ;\\  
                                                          \PP^{r}\times\PP^{r-2}&\hbox{if\ }p= E_{2,4}\ ;\\
                                                          \PP^{r-2}\times\PP^{r-1}&\hbox{if\ }p\in E_{1,1}^0\ ;\\
                                                           \PP^{r-1}\times\PP^{r-2}&\hbox{if\ }p\in E_{1,2}^0\ ;\\ 
                                                   \PP^{r-1}\times\PP^{r-2}&\hbox{if\ }p\in E_{1,4}^0\ ;\\    
                                        \PP^{r-2}\times\PP^{r-1}&\hbox{if\ }p\in E_{1,3}^0\ ;\\
                                        \PP^{r-2}\times\PP^{r-2}&\hbox{if\ }p\in E^0\ .\\
                                                                  \end{cases}\]
                                                   \end{lemma}
                                                   
   \begin{proof} We consider a point on $E$ is given by the data
    \[   \bigl\{    \bigl(x,t\bigr)\in \PP\times \PP^3\bigr\}\ ,\]
    where $(x,x)$ is a point on the diagonal, and $L_t$ is a line in $\PP$ passing through $x$.
   The point $x$ imposes $1$ condition on the $G_b$, except for the point $[0:0:0:1:0]$ which imposes no condition on the $G_b$, cf. Lemma \ref{one}. The point $x$ imposes one condition on the $F_b$, except for the point $[0:0:0:0:1]$ which imposes no condition on the $F_b$.

Suppose now the point $x$ is not one of the exceptional points $[0:0:0:1:0], [0:0:0:0:1]$, i.e. we are outside of $E_{1,1}\cup E_{1,2}$. Suppose also the line $L_t$ is not contained in $\Delta_{3,+,+}$, i.e. we are outside of $E_{1,3}$. We consider the morphism
  \[ \phi\colon\ \ \PP\setminus [0:0:0:1:0]\ \xrightarrow{\phi_1}\ \PP(1,1,1,2)  \ \xrightarrow{\phi_2}\ \PP^\prime:=\PP(2,1,1,4)\ ,\]
  where $\phi_1$ is obtained by forgetting the $z_3$ coordinate, and $\phi_2$ is obtained by letting groups of roots of unity act diagonally. The image $\phi(L_t)$ is a line passing through the point $\phi(x)$. Since the line bundle ${\mathcal O}_{\PP^\prime}(4)$ is very ample (cf. Lemma \ref{delorme} below), there exists a polynomial $g$ of weighted degree $4$ such that the hypersurface $(g=0)$ contains $\phi(x)$ and is transverse to $\phi(L_t)$. The polynomial $g$ looks like
    \[ \lambda z_4 + c^\prime w^2 + w(\hbox{something\ quadratic\ in\ }x_1,x_2)+(\hbox{something\ quartic\ in\ }x_1,x_2)\ .\]
    The inverse image $\phi^{-1}(g=0)$ in $\PP$ looks like
    \[   \lambda z_4^4 + c^\prime w^4 + w^2(\hbox{something\ quadratic\ in\ }x_1,x_2)+(\hbox{something\ quartic\ in\ }x_1,x_2)=0\ ,\]
    that is we have found a $G_b$, $b\in\bar{B}$ containing $x$ and transverse to the line $L_t$. This shows that a point $p\in E^0$ imposes $2$ independent conditions on the polynomials $G_b$. Since the argument with respect to the $F_b$ and $\Delta_{4,+,+}$ is symmetric, this proves the last line.

    Suppose now $x\not\in \{ [0:0:0:1:0], [0:0:0:0:1]\}$ and $L_t\subset\Delta_{3,+,+}$, i.e. $p\in E_{1,3}^0$. The line $L_t$ disappears under the projection $\phi_1$, which means that all the $G_b$ will be tangent to $L_t$; this proves the one--but--last line. 
    
    The remaining cases are similarly checked. For instance, suppose $x=[0:0:0:0:1]$ and $L_t\not\subset\Delta_{3,+,+}$, i.e. the point $p$ is in $E_{1,2}^0\cup E_{2,4}$.
    The line $L_t$ does not disappear under the projection $\phi$, so (as above) the point $p$ imposes $2$ conditions on the $G_b$. If $L_t\not\subset\Delta_{4,+,+}$, the point $p$ imposes $1$ condition on the $F_b$.\footnote{We remark that a detailed analysis of $E_{1,1}$ and $E_{1,2}$ is not absolutely necessary to our argument; an easy way out is as follows: the dimension of $M$ is $\dim(\PP\times\PP)+2(r-2)=2r+4$. The dimension of $\pi^{-1}(E_{1,1}\cup E_{1,2})$ is (by what we have said above) at most $3+r-2+r=2r+1$, so whatever happens above $E_{1,1}\cup E_{1,2}$ can not interfere with codimension $2$ cycles: we have 
      \[A_{s-2}(M)\cong A_{s-2}(M\setminus \pi^{-1}(E_{1,1}\cup E_{1,2}))\ \]
      (where $s:=\dim M$). 
      That is: as long as we are only interested in codimension $2$ cycles, we may just as well leave out $E_{1,1}$ and $E_{1,2}$.}
     \end{proof}

    \begin{lemma}\label{delorme} Let $P^\prime$ be the weighted projective space $\PP(2,1,1,4)$. Then the line bundle ${\mathcal O}_{P^\prime}(4)$ is very ample.
    \end{lemma}
    
    \begin{proof} The coherent sheaf ${\mathcal O}_{P^\prime}(4)$ is locally free, because $4$ is a multiple of the ``weights'' \cite{Dol}. To see that this line bundle is very ample, we use the following numerical criterion: 
 
 \begin{proposition}[Delorme \cite{Del}]\label{del} Let $P=\PP(q_0, q_1,\ldots,q_n)$ be a weighted projective space. Let $m$ be the least common multiple of the $q_j$. Suppose every monomial
 \[x_0^{b_0} x_1^{b_1}\cdots x_n^{b_n}\]
 of (weighted) degree $km$ ($k\in \NN^\ast$) is divisible by a monomial of (weighted) degree $m$. Then ${\mathcal O}_{P}(m)$ is very ample.
 \end{proposition}
 
(This is the case $E(x)=0$ of \cite[Proposition 2.3(\rom3)]{Del}.)

Lemma \ref{delorme} is now easily established: suppose 
  \[x^{\underline{b}}=x_0^{b_0}x_1^{b_1}x_2^{b_2}x_3^{b_3}\] 
  is a monomial of weighted degree $4k$, i.e.
  \[ 2b_0+b_1+b_2+4b_3=4k\ .\]
  If $b_3\ge 1$, then $x_3$ divides $x^{\underline{b}}$ and we are OK. Suppose now $b_3=0$. If $b_0\ge 2$, then we are OK since $x_0^2$ divides $x^{\underline{b}}$. If $b_0=1$, then $b_1+b_2\ge 2$ and $x_0$ times something quadratic ($x_1^2$ or $x_1x_2$ or $x_2^2$) divides $x^{\underline{b}}$. The remaining case $b_3=b_0=0$ is obviously OK.
  \end{proof}

   \begin{corollary}\label{strong2} The variety $M_E$ has the strong property.
   \end{corollary}
   
   \begin{proof} This follows from Lemma \ref{projbundle} and Lemma \ref{local}.
   \end{proof}                                                
                                                   
   Now we are able to wrap up the proof: the variety $M$ that we are interested in is a disjoint union of three strata
     \[   M=M^0\cup M_Q\cup M_E\ .\]
     Each of these three strata
       has the strong property (Corollaries \ref{strong} and \ref{strong3} and \ref{strong2}); applying Lemma \ref{local}, it follows that $M$ has the strong property, i.e.
       \[ A_\ast^{hom}(M)_{\QQ}=0\ ,\]
   which proves Proposition \ref{key2}.        
        \end{proof}

 \begin{remark}\label{vsc} If we assume the ``Voisin standard conjecture'' ( \cite[Conjecture 0.6]{V0}, \cite[Conjecture 2.29]{Vo})  is true, we obtain the stronger statement that
   \[  A^i_{hom}(\VV^\prime\times_{B^\prime} \VV^\prime)_{\QQ}=0\ \]
   for all $i$. Since for our argument, we are only interested in (surfaces and hence) codimension $2$ cycles on $\VV^\prime\times_{B^\prime} \VV^\prime$, 
   we have no need for this conditional stronger statement. As noted in \cite{V0}, the Voisin standard conjecture is true in codimension $2$ and so in this case one obtains an unconditional statement.
    \end{remark}

\begin{remark}\label{integral} We note in passing that 
everything we say in this section is still valid when replacing "the weak (resp. strong) property" by "the weak (resp. strong) Chow--K\"unneth property". This last notion is defined in  \cite[page 10]{T} (and further studied in \cite{T2}). This implies (using \cite[Proposition 2]{T} or \cite{FMSS}) that the variety $M$ satisfies
  \[ A^i(M)\cong \hbox{Hom}(A_i(M),\ZZ)\ \ \hbox{for\ all\ i}\ ,\]
  where the left--hand side denotes Fulton--MacPherson's operational Chow cohomology \cite{F}.
We do not need this statement here.
 \end{remark}

\section{Motives}

This section contains a motivic version of the main result, stating that for the Todorov surfaces under consideration, the ``transcendental part of the motive'' (in the sense of \cite{KMP}) is isomorphic to the transcendental part of the motive of the associated $K3$ surface (Theorem \ref{main2}). Some consequences are given.

\begin{theorem}[Kahn--Murre--Pedrini \cite{KMP}]\label{trans} Let $S$ be any smooth projective surface, and let $h(X)\in\MM_{\rm rat}$ denote the Chow motive of $S$.
There exists a self--dual Chow--K\"unneth decomposition $\{\pi_i\}$ of $S$, with the property that there is a further splitting in orthogonal idempotents
  \[ \pi_2= \pi_2^{alg}+\pi_2^{tr}\ \ \hbox{in}\ A^2(S\times S)_{\QQ}\ .\]
  The action on cohomology is
  \[  (\pi_2^{alg})_\ast H^\ast(S)= N^1 H^2(S)\ ,\ \ (\pi_2^{tr})_\ast H^\ast(S) = H^2_{tr}(S)\ ,\]
  where the transcendental cohomology $H^2_{tr}(S)\subset H^2(S)$ is defined as the orthogonal complement of $N^1 H^2(S)$ with respect to the intersection pairing. The action on Chow groups is
  \[ (\pi_2^{alg})_\ast A^\ast(S)_{\QQ}= N^1 H^2(S)\ ,\ \ (\pi_2^{tr})_\ast A^\ast(S) = A^2_{AJ}(S)_{\QQ}\ .\]  
 This gives rise to a well--defined Chow motive
  \[ t_2(S):= (S,\pi_2^{tr},0)\ \subset \ h(X)\ \ \in\MM_{\rm rat}\ ,\]
  the so--called {\em transcendental part of the motive of $S$}.
  \end{theorem}

\begin{proof} Let $\{\pi_i\}$ be a Chow--K\"unneth decomposition as in \cite[Proposition 7.2.1]{KMP}. The assertion then follows from \cite[Proposition 7.2.3]{KMP}.
\end{proof}

\begin{theorem}\label{main2} Let $S$ be a Todorov surface with $K^2_S=2$ and $\pi_1(S)=\ZZ/ 2\ZZ$, and let $P$ be the $K3$ surface obtained as a resolution of singularities of $S/\iota$.
The natural correspondence from $S$ to $P$ induces an isomorphism of Chow motives
  \[  t_2(S)\ \cong\ t_2(P)\ \ \hbox{in}\ \MM_{\rm rat}\ .\]
 \end{theorem}

\begin{proof} This is just a dressed--up version of the argument of Theorem \ref{main}. Let $\Ss\to B$ be the family of canonical models of Todorov surfaces with fundamental invariants $(1,10)$, as in Corollary \ref{family}. As in the proof of Theorem \ref{main}, we have morphisms of families over $B$
    \[ \begin{array}[c]{ccc}  
                                   \wt{\Ss}& \to & \Ss\\
                                  \ \ \ \  \downarrow{\wt{f}}&& \ \ \downarrow{f}\\
                                  \wt{\MM}&\to& \MM\\
                                   &\searrow\ &\downarrow\\
                                  &&B\\
                                  \end{array}\]

Here $\MM$ is the family of $K3$ surfaces with rational double points, and $\wt{\MM}$ is the family of desingularized $K3$ surfaces.
Taking the graph of the morphism $\wt{f}$, one gets a relative correspondence
  \[ \Gamma_{\wt{f}}\in A_{s-2}(\wt{\Ss}\times_B \wt{\MM})\ \]
  (here $s$ denotes $\dim \wt{\Ss}\times_B \wt{\MM}$). 
  The proof of Theorem \ref{main}, applied to the relative correspondence
   \[ 2 \Delta_{\wt{\Ss}} -{}^t \Gamma_{\wt{f}}\circ \Gamma_{\wt{f}}\ \ \in A_{s-2}(\wt{\Ss}\times_B \wt{\Ss}) \] 
     gives a rational equivalence for the general (and hence, for any) $b\in B$:
      \begin{equation}\label{rat}  
          2\Delta_{\wt{S}_b}= {}^t \Gamma_{\wt{f}_b} \circ \Gamma_{\wt{f}_b} + \sum_{i,j} D_i\times D_j+\gamma \ \ \in A^2(\wt{S}_b\times \wt{S}_b)_{\QQ}\ ,\end{equation}
  where the $D_i\subset \wt{S}_b$ are divisors, and $\gamma$ is supported on
  \[  E_b\times\wt{S}_b\ \cup\ \wt{S}_b\times E_b\ ,\]
  and $E_b\subset\wt{S}_b$ is an exceptional divisor for the morphism $\wt{S}_b\to S_b$. Note that $\gamma$ is contained in the ideal of so--called ``degenerate correspondences'' 
  ${\mathcal J}(\wt{S}_b,\wt{S}_b)$ \cite[page 309]{F}, \cite[Definition 7.4.2]{KMP}.
  
  Consider now $S$ a Todorov surface with fundamental invariants $(1,10)$, and $P$ the associated $K3$ surface.
  It follows from Corollary \ref{family} that the canonical model of $S$ is an $S_b$ for some $b\in B$, so that $S$ is birational to the smooth surface $\wt{S}_b$ and $P$ is birational to the smooth surface $\wt{\MM}_b$.
    Let 
    \[\pi_0^S, \ \pi_2^S, \ \pi_4^S\ ,\ \  \hbox{and}\ \  \pi_0^M,\  \pi_2^M,\ \pi_4^M\] 
    denote a Chow--K\"unneth decomposition for $\wt{S}_b$, resp. for $\wt{\MM}_b$, as in Theorem \ref{trans}. Let 
    \[   \pi_2^S= \pi_2^{S,alg} +\pi_2^{S,tr}\ \ \hbox{and}\ \ 
                       \pi_2^M= \pi_2^{M,alg} +\pi_2^{M,tr}\  
                       \]
                     be the refined decomposition of Theorem \ref{trans}.                     
                     Since $\pi_2^{S,alg}$ is a projector on $NS(\wt{S}_b)_{\QQ}$, we have equality
                     \[  \sum_{i,j} D_i\times D_j=(\sum_{i,j} D_i\times D_j)\circ \pi_2^{S,alg}\ ,\]
                  and hence (since $\pi_2^{S,tr}$ and $\pi_2^{S,alg}$ are orthogonal) we find that
                   \[  (\sum_{i,j} D_i\times D_j)\circ \pi_2^{S,tr}=0\ \ \hbox{in}\ A^2(\wt{S}_b\times \wt{S}_b)_{\QQ}\ .\]
              Likewise, since $\gamma\in {\mathcal J}(\wt{S}_b,\wt{S}_b)$, it follows from \cite[Theorem 7.4.3]{KMP} that
                 \[   \pi_2^{S,tr}\circ \gamma \circ \pi_2^{S,tr}=0\ \ \hbox{in}\ A^2(\wt{S}_b\times \wt{S}_b)_{\QQ}\ .\]
                 
               It now follows from equation \ref{rat} (after twice applying $\pi_2^{S,tr}$ on both sides) that                    
                  \[  2 \pi_2^{S,tr}=  \pi_2^{S,tr}\circ {}^t \Gamma_{\wt{f}_b} \circ \Gamma_{\wt{f}_b}\circ \pi_2^{S,tr}\ \ \hbox{in}\ A^2(\wt{S}_b\times \wt{S}_b)_{\QQ}\ .\]
      
      The next step is to remark that
        \[  \begin{split}   &\pi_4^M\circ \Gamma_{\wt{f}_b} \circ \pi_2^S=0\ ,\\            
                                  &\pi_2^S\circ {}^t \Gamma_{\wt{f}_b} \circ \pi_0^M=0\ ,\\
                                  \end{split}\]
          (this follows from \cite[Theorem 7.3.10 (\rom1)]{KMP}), and so we have
             \[ 2\pi_2^{S,tr}\circ {}^t \Gamma_{\wt{f}_b} \circ \Gamma_{\wt{f}_b}\circ \pi_2^{S,tr}=  \pi_2^{S,tr}\circ {}^t \Gamma_{\wt{f}_b} 
             \circ \pi_2^M\circ \Gamma_{\wt{f}_b}\circ \pi_2^{S,tr}  \ .\] 
             %
            But       
            \[ \pi_2^M\circ \Gamma\circ \pi_2^{S,tr}  =  \pi_2^{M,tr}\circ \Gamma\circ \pi_2^{S,tr}\]
           \cite[Lemma 7.4.1]{KMP}, so we end up with a rational equivalence
           \[ 2 \pi_2^{S,tr}=  \pi_2^{S,tr}\circ {}^t \Gamma_{\wt{f}_b} \circ \pi_2^{M,tr}\circ \Gamma_{\wt{f}_b}\circ \pi_2^{S,tr}\ \ \hbox{in}\ A^2(\wt{S}_b\times \wt{S}_b)_{\QQ}\ .\]
           
           The fact that there is also a rational equivalence
           \[ 2 \pi_2^{M,tr}=  \pi_2^{M,tr}\circ {} \Gamma_{\wt{f}_b} \circ \pi_2^{S,tr}\circ {}^t \Gamma_{\wt{f}_b}\circ \pi_2^{M,tr}\ \ \hbox{in}\ A^2(\wt{\MM}_b\times \wt{\MM}_b)_{\QQ}\ \]
           is much easier: we have
           \[ 2\Delta_{\wt{\MM}_b} =  {} \Gamma_{\wt{f}_b} \circ {}^t \Gamma_{\wt{f}_b}\ \ \hbox{in}\ A^2(\wt{\MM}_b\times \wt{\MM}_b)_{\QQ}\ ,\]
           and the same argument applies.
           
           We have now established that
           \[ \Gamma_{\wt{f}_b}\colon\ \ t_2(\wt{S}_b)\ \to\ t_2(\wt{\MM}_b) \ \ \hbox{in}\ \MM_{\rm rat}\]
           is an isomorphism of motives, with inverse given by ${}^t \Gamma_{\wt{f}_b}$.      
                
           Since the surfaces $S$ and $P$ are birational to $\wt{S}_b$ resp. to $\wt{\MM}_b$, and $t_2$ is a birational invariant amongst smooth surfaces, it follows that there is also an isomorphism
           \[ t_2(S)\cong t_2(P)\ \ \hbox{in}\ \MM_{\rm rat}\ .\]
           
               \end{proof}

\begin{corollary}\label{finite} Let $S$ be a Todorov surface with $K^2_S=2$ and $\pi_1(S)=\ZZ/ 2\ZZ$. Assume moreover that one of the following holds:

\noindent
(\rom1) the Picard number $\rho(S)$ is at least $h^{1,1}(S)-1$;

\noindent
(\rom2) the $K3$ surface birational to $S/\iota$ is a Kummer surface;

\noindent
(\rom3) the $K3$ surface birational to $S/\iota$ has a Shioda--Inose structure (in the sense of \cite{Mor1}).

Then $S$ has finite--dimensional motive (in the sense of Kimura and O'Sullivan \cite{K}, \cite{An}).
\end{corollary}

\begin{proof} It suffices to show that the motive $t_2(S)$ is finite--dimensional, hence (applying Theorem \ref{main2}) that the motive $t_2(P)$ is finite--dimensional, where $P$ is the $K3$ surface birational to $S/\iota$. In case (\rom1), this is true since
  \[ H^2_{tr}(P)\cong H^2_{tr}(S)\]
  has dimension $\le 3$, so that the Picard number $\rho(P)$ is $\ge 19$, and $K3$ surfaces with Picard number $\ge 19$ are known to have finite--dimensional motive \cite{P}. In case (\rom2), the needed statement is obviously true; in case (\rom3) it follows from \cite[Remark 47]{thoughts}.
  \end{proof}

 \begin{remark} It should be noted that Todorov surfaces as in Corollary \ref{finite} (\rom2) can be readily constructed; in fact, these were the first examples given by Todorov 
 \cite{Tod2}.
 \end{remark}

\begin{remark} We note in passing that surfaces $S$ as in Corollary \ref{finite} not only have finite--dimensional motive; their motive is actually in the subcategory of motives of abelian type (that is, the category of Chow motives generated by the motives of curves \cite{V3}). The same is true for $K3$ surfaces: all examples of $K3$ surfaces known to be finite--dimensional are actually of abelian type.

This is not surprising, for the following reason: for {\em any\/} surface $S$ with $p_g(S)=1$, the Kuga--Satake construction \cite{KS} relates $H^2(S)$ to the cohomology of an abelian variety. If the Kuga--Satake correspondence is algebraic (e.g., if the Hodge conjecture is true), this means that the homological motive of $S$ is direct factor of the motive of an abelian variety plus a sum of curves. If $S$ has finite--dimensional motive, the same is true for the Chow motive of $S$, i.e. $S$ has motive of abelian type.
\end{remark}

\begin{corollary} Let $S, S^\prime$ be two Todorov surfaces as in Corollary \ref{finite}. Assume there exists a Hodge isometry between the transcendental lattices
  \[ \phi\colon\ \ T_S\otimes\QQ\ \cong\ T_{S^\prime}\otimes\QQ \]
  (i.e., $\phi$ is an isomorphism of Hodge structures that respects the intersection forms). Then there is an isomorphism of motives
    \[ t_2(S)\ \cong\ t_2(S^\prime)\ \ \hbox{in}\ \MM_{\rm rat}\ .\]
    \end{corollary}
    
    \begin{proof} Let $P, P^\prime$ denote the associated $K3$ surfaces. Then $\phi$ induces a Hodge isometry
    \[  T_P\otimes\QQ\ \cong\ T_{P^\prime}\otimes\QQ\ \]
    (since in both cases, the intersection form is multiplied by $2$ when going to the double cover).
    By Mukai \cite{Mu}, this Hodge isometry is induced by a cycle $\Gamma\in A^2(P\times P^\prime)_{\QQ}$ (note that the assumptions of Corollary \ref{finite} imply $P$ and $P^\prime$ have Picard number $\ge 17$, so \cite{Mu} indeed applies). Then $\Gamma$ induces an isomorphism of homological motives
    \[ \Gamma\colon\ \ t_2(P)\ \cong\ t_2(P^\prime)\ \ \hbox{in}\ \MM_{\rm hom}\ ,\]
and hence    
    (using finite--dimensionality of $P$ and $P^\prime$) an isomorphism of Chow motives
    \[   \Gamma\colon\ \ t_2(P)\ \cong\ t_2(P^\prime)\ \ \hbox{in}\ \MM_{\rm rat}\ .\]
    The corollary now follows by combining with Theorem \ref{main2}.
    \end{proof}

Another corollary is that a weak form of the relative Bloch conjecture is true for surfaces as in Corollary \ref{finite}:

\begin{corollary} Let $S$ be a Todorov surface as in Corollary \ref{finite}. Let $\Gamma\in A^2(S\times S)_{\QQ}$ be a correspondence such that
  \[ \Gamma_\ast=\hbox{id}\colon\ \ H^{2,0}(S)\ \to\ H^{2,0}(S)\ .\]
  Then 
  \[  \Gamma_\ast \colon\ \ A^2_{hom}(S)_{\QQ}\ \to\ A^2_{hom}(S)_{\QQ} \]
  is an isomorphism.
  \end{corollary}
  
  \begin{proof} As is well--known, this holds for any surface $S$ with finite--dimensional motive.
  \end{proof}

\section{Speculation}

This final section offers some speculation about possible directions of generalization of the results in this note.

\begin{remark} It would be interesting to try and prove Corollary \ref{true} for all surfaces $S$ with $p_g=1$, $K_S^2=2$ and $\pi_1(S)=\ZZ/2\ZZ$. Thanks to Catanese--Debarre \cite{CD}, (canonical models of) these surfaces form a $16$--dimensional family, explicitly described as quotients of complete intersections in $\PP(1,1,1,2,2)$. 

The problem is that outside of the $12$--dimensional Todorov locus (where there is a $K3$ surface over which $S$ is a double cover), it seems difficult to exploit this fact. An argument such as sketched by Voisin for quartic surfaces \cite[Theorem 3.10]{V0} would perhaps work to establish Corollary \ref{true} for this $16$--dimensional family, but this argument is conditional on (1) knowing the generalized Hodge conjecture for 
  \[\wedge^2 H^2(S)\subset H^4(S\times S)\ ,\] 
  and (2) knowing the ``Voisin standard conjecture'' \cite[Conjecture 0.6 ]{V0}, \cite[Conjecture 2.29]{Vo} is true (to obtain a cycle supported on a certain subvariety).
  \end{remark}

\begin{remark}
On the other hand, it would also be interesting to prove Corollary \ref{true} for the $9$ other families of Todorov surfaces. Thanks to the work of Rito (cf. \cite{Rito}, or Theorem \ref{rito}), these have an associated $K3$ surface for which Voisin's conjecture is known. 

The problem is in relating $0$--cycles on $S$ to $0$--cycles on the associated $K3$ surface (that is, in proving Theorem \ref{main}). What is needed at the very least, in order for the ``spreading out'' approach of \cite{V0}, \cite{V1} to work, is that the irreducible family
$\Ss\to B$ of Todorov surfaces with given fundamental invariants is nice enough to have the property that
  \[ A^2_{hom}(\Ss\times_B \Ss)_{\QQ}=0\ .\]
 However, in the absence of an explicit description of the family (such as given by the weighted complete intersections of \cite{CD} in case $K_S^2=2$ and $\pi_1(S)=\ZZ/2\ZZ$), this seems difficult. Can this property perhaps be proven for the total space of the deformation of \cite{LP} mentioned in Remark \ref{remtod} ?
\end{remark}

\vskip0.8cm

\begin{nonumberingt} Thanks to Claire Voisin for having written the invaluable and inspirational monograph \cite{Vo}. Thanks to the referee for helpful suggestions, that significantly improved the exposition in several places.
Many thanks (as always) to Yasuyo, Kai and Len, who provide wonderful working conditions in Schiltigheim.
\end{nonumberingt}

\end{document}